\documentclass{amsart}
\usepackage{graphics}
\usepackage{graphicx}
\usepackage{amsfonts,amsmath,mathrsfs}
\begin{document}

 \newtheorem{theorem}{Theorem}[section]
 \newtheorem{corollary}[theorem]{Corollary}
 \newtheorem{lemma}[theorem]{Lemma}{\rm}
 \newtheorem{proposition}[theorem]{Proposition}

 \newtheorem{defn}[theorem]{Definition}{\rm}
 \newtheorem{assumption}[theorem]{Assumption}
 \newtheorem{remark}[theorem]{Remark}
 \newtheorem{ex}{Example}
\numberwithin{equation}{section}
\def\la{\langle}
\def\ra{\rangle}
\def\glexe{\leq_{gl}\,}
\def\glex{<_{gl}\,}
\def\e{{\rm e}}

\def\fac{{\rm !}}
\def\x{\mathbf{x}}
\def\P{\mathbb{P}}
\def\S{\mathbf{S}}
\def\h{\mathbf{h}}
\def\m{\mathbf{m}}
\def\y{\mathbf{y}}
\def\bz{\mathbf{z}}
\def\F{\mathcal{F}}
\def\R{\mathbb{R}}
\def\T{\mathbf{T}}
\def\N{\mathbb{N}}
\def\D{\mathbf{D}}
\def\V{\mathbf{V}}
\def\U{\mathbf{U}}
\def\K{\mathbf{K}}
\def\Q{\mathbf{Q}}
\def\M{\mathbf{M}}
\def\oM{\overline{\mathbf{M}}}
\def\O{\mathbf{O}}
\def\C{\mathbb{C}}
\def\P{\mathbb{P}}
\def\Z{\mathbb{Z}}
\def\bZ{\mathbf{Z}}
\def\H{\mathcal{H}}
\def\A{\mathbf{A}}
\def\V{\mathbf{V}}
\def\AA{\overline{\mathbf{A}}}
\def\B{\mathbf{B}}
\def\c{\mathbf{C}}
\def\L{\mathcal{L}}
\def\bS{\mathbf{S}}
\def\H{\mathcal{H}}
\def\I{\mathbf{I}}
\def\Y{\mathbf{Y}}
\def\X{\mathbf{X}}
\def\cX{\mathbf{X}}
\def\G{\mathbf{G}}
\def\f{\mathbf{f}}
\def\z{\mathbf{z}}
\def\v{\mathbf{v}}
\def\y{\mathbf{y}}
\def\d{\hat{d}}
\def\x{\mathbf{x}}
\def\bI{\mathbf{I}}
\def\y{\mathbf{y}}
\def\g{\mathbf{g}}
\def\w{\mathbf{w}}
\def\b{\mathbf{b}}
\def\a{\mathbf{a}}
\def\p{\mathbf{p}}
\def\u{\mathbf{u}}
\def\bv{\mathbf{v}}
\def\q{\mathbf{q}}
\def\e{\mathbf{e}}
\def\s{\mathcal{S}}
\def\cc{\mathcal{C}}
\def\co{{\rm co}\,}
\def\tg{\tilde{g}}
\def\tx{\tilde{\x}}
\def\tg{\tilde{g}}
\def\tA{\tilde{\A}}
\def\bell{\boldsymbol{\ell}}
\def\bxi{\boldsymbol{\xi}}
\def\balpha{\boldsymbol{\alpha}}
\def\bbeta{\boldsymbol{\beta}}
\def\bgamma{\boldsymbol{\gamma}}
\def\bpsi{\boldsymbol{\psi}}
\def\supmu{{\rm supp}\,\mu}
\def\supp{{\rm supp}\,}
\def\cd{\mathcal{C}_d}
\def\cok{\mathcal{C}_{\K}}
\def\cop{COP}
\def\vol{{\rm vol}\,}
\def\om{\mathbf{\Omega}}
\def\f{\mathscr{F}}
\def\la{\langle\langle}
\def\ra{\rangle\rangle}
\def\blambda{\boldsymbol{\lambda}}
\def\btheta{\boldsymbol{\theta}}
\def\bphi{\boldsymbol{\phi}}
\def\bpsi{\boldsymbol{\psi}}
\def\bnu{\boldsymbol{\nu}}
\def\bom{\boldsymbol{\Omega}}
\def\fac{{\rm !}}
\def\tM{\hat{\M}}
\def\tv{\hat{\v}}

\title[Classification]{On the Christoffel function and classification in data analysis}
\thanks{Research supported by the AI Interdisciplinary Institute ANITI  funding through the french program
\emph{Investing for the Future PI3A} under the grant agreement number ANR-19-PI3A-0004. The author is also affiliated with IPAL-CNRS laboratory, Singapore.}

\author{Jean B. Lasserre}
\address{LAAS-CNRS and Institute of Mathematics\\
University of Toulouse\\
LAAS, 7 avenue du Colonel Roche\\
31077 Toulouse C\'edex 4, France\\
Tel: +33561336415}
\email{lasserre@laas.fr}

\date{}
\begin{abstract}
We show that the empirical Christoffel function associated with a cloud of
 finitely many points sampled from a distribution, 
 can provide a simple tool for supervised classification in data analysis, with good generalization properties.
\end{abstract}

\maketitle

\section{Introduction}
In this note we are mainly concerned with supervised classification with noiseless deterministic labels
where the objects of interest $\x\in\X$ belong to $m$ classes with supports
$\X_j\subset\X\subset\R^n$, $j\in [m]$ (with $[m]=\{1,\ldots,m\}=:\Y$). 
The supports satisfy $\X_i\cap\X_j=\emptyset$ for all $i,j$ with $i\neq j$. 
The data set consists of clouds of finitely many points $(\x(i))\subset \X_j$ sampled from an underlying distribution $\phi_j$ on $\X_j$, $j\in [m]$. 
In this situation, an \emph{exact} classifier $f:\X\to\Y$, selects $j=:f(\x)$ whenever $\x\in\X_j$. 
When constructing a classifier from a sample of  data points, as e.g. in machine learning, a sensitive issue is
its  generalization properties when applied on a test set different from the training set.
For the reader interested  in recent developments on various techniques and issues in supervised and unsupervised 
classification, we refer to e.g. the book \cite{data-driven} and the many references therein.

\paragraph{Contribution} We first introduce a simple and natural ideal classifier $f_t:\X\to \Y$,
with nice asymptotic properties as $t$ increases. It is based on the 
Christoffel function $\Lambda^\mu_t$ associated with the joint distribution 
$d\mu(\x,y)$ on $\X\times\Y$. As $\mu$ is supported on the graph
$\{(\x,f(\x)):\x\in\X\,\}$ of the exact classifier $f$, recent
results of \cite{constructive} transported in our context
suggest that the classifier $f_t(\x):=\arg\max_{y\in \Y}\Lambda^\mu_t(\x,y)$
should approximate $f$ nicely. This is the case and indeed, by a slight modification of the definition of 
$\Lambda^\mu_t$, we show that $f_t$ is simply expressed in terms of the 
Christoffel functions $\Lambda^{\phi_j}_t$ of the $\phi_j$; namely,
$f_t(\x)=\arg\max_k\Lambda^{\phi_k}_t(\x)$. Notice that this simple form 
of $f_t$ mathematically justifies for supervised classification,
the intuitive argument that $\Lambda^{\phi_j}_t(\x)>\Lambda^{\phi_k}_t(\x)$, $\forall k\neq j$, whenever $t$ is sufficiently large and $\x\in\X_j$. 
Indeed as $\x\in \X_j$ is outside the support $\X_k$ of $\phi_k$, for every $k\neq j$,
the ``score" $\Lambda^{\phi_k}_t(\x)$ is close to zero
for sufficiently large $t$, as it decreases exponentially fast to zero 
(while the decrease of $\Lambda^{\phi_j}_t(\x)$ is at most polynomial in $t$).

We next consider the practical case where we only have access 
to a discrete sample of points in each class $\X_j$ (e.g., the training set in Machine Learning)
so that $\Lambda^{\phi_j}_t$ is not available. We provide a \emph{data-driven} analogue of
the previous result which, as expected, is in terms of the Christoffel functions $\Lambda^{\phi_{j,N}}_t$
associated with the discrete empirical measures $\phi_{j,N}$. Namely the empirical discrete analogue 
$f^N_t$ of the classifier $f_t$ simply reads $f^N_t(\x)=\arg\max_k\Lambda^{\phi_{k,N}}_t(\x)$, and 
has same properties  as $f_t$, but of course in an almost-sure sense with respect to random samples. 
In particular it shows good generalization properties. Indeed with $\varepsilon>0$ fixed and
$t$ sufficiently large, with probability $1$ (with respect to random samples), 
$f^N_t(\x)=j$ for every $j\in [m]$ and all $\x\in\X_j$ at distance at least $\varepsilon$ from the boundary $\partial\X_j$,
for sufficiently large $N$.

Finally, we  also briefly discuss more general joint distributions of pairs $(\x,y)$
(where possibly $\X_i\cap\X_j\neq\emptyset$  for some  $(i,j)$)
which covers practical cases where some misclassification may occur and/or some ambiguity is allowed.
 \subsection{Notation, definitions and preliminary results}
Let $\R[\x]$ denote the ring of real polynomials in the variables $\x=(x_1,\ldots,x_n)$ and $\R[\x]_t\subset\R[\x]$ be its subset 
of polynomials of total degree at most $t$. Let $\N^n_t:=\{\balpha\in\N^n:\vert\balpha\vert\leq t\}$
(where $\vert\balpha\vert=\sum_i\alpha_i$) with cardinal $s(t)={n+t\choose n}$. Let $\bv_t(\x)=(\x^{\balpha})_{\balpha\in\N^n_t}$ 
be the vector of monomials up to degree $t$.

The support of a Borel measure $\mu$ on $\R^n$ is the smallest closed set $A$ such that
$\mu(\R^n\setminus A)=0$, and such a  set $A$ is unique.

\paragraph{Moment matrix}
Let $\phi$ be a Borel measure whose support $\bom\subset\R^n$ is compact with nonempty interior. Its 
\emph{moment matrix} of order (or degree) $t$ is the real symmetric matrix $\M_t(\phi)$
with rows and columns indexed by $\N^n_t$, and with entries
\[\M_t(\phi)(\balpha,\bbeta)\,:=\,\int_{\bom}\x^{\balpha+\bbeta}\,d\phi\,=\,\phi_{\balpha+\bbeta}\,,\quad\balpha,\bbeta\in\N^n_t\,.\]
Then necessarily $\M_t$ is positive semidefinite for all $t$, denoted $\M_t(\phi)\succeq0$. 

\paragraph{Christoffel function} 
If $\bom$ has nonempty interior then $\M_t$ is positive definite for all $t$, denoted $\M_t(\phi)\succ0$. 
Let $(P_{\balpha})_{\balpha\in\N^n}\subset\R[\x]$ be a family of polynomials, orthonormal with respect to $\phi$, i.e.,
\[\int_{\bom}P_{\balpha}\,P_{\bbeta}\,d\phi\,=\,\delta_{\balpha=\bbeta}\,,\quad\forall \balpha,\bbeta\in\N^n\,.\]
Then the 
Christoffel function (CF) $\Lambda^{\phi}_t:\R^n\to\R_+$  associated with $\phi$, is defined by
\begin{equation}
\label{def-christo-1}
\x\mapsto \Lambda^\phi_t(\x)\,:=\,\left[\sum_{\balpha\in\N^n_t} P_{\balpha}(\x)^2\right]^{-1}\,,\quad\forall \x\in\R^n\,,\end{equation}
and recalling that $\M_t(\mu)$ is nonsingular, it turns out that 
\begin{equation}
\label{def-christo-11}
\Lambda^\phi_t(\x)\,=\,\left[\,\v_t(\x)^T\,\M_t(\phi)^{-1}\,\v_t(\x)\,\right]^{-1}\,,\quad\forall \x\in\R^n\,.
\end{equation}
An equivalent and variational definition is also
\begin{equation}
\label{def-christo-2}
\Lambda^\phi_t(\x)\,=\,\inf_{p\in\R[\x]_t}\{\,\int_{\bom} p^2\,d\phi\::\: p(\x)\,=\,1\,\}\, ,\quad\forall \x\in\R^n\,.\end{equation}
One interesting and distinguishing feature of the CF is that as $t$ increases, $\Lambda^\phi_t(\x)\downarrow 0$ exponentially fast for every $\x\not\in{\rm support}(\phi)$. In other words $\Lambda_t^\phi$ identifies the support of $\phi$ when $t$ is sufficiently large. In addition, at least in dimension $n=2$ or $n=3$, one may visualize this property even for small $t$, as the resulting superlevel sets 
$\bom_\gamma:=\{\,\x: \Lambda^\phi_t(\x)\geq \gamma\,\}$, $\gamma\in\R$, capture the shape of $\bom$ quite well; see e.g. \cite{neurips}.

\subsection{Setting}

Let $\Y:=\{1,2,\ldots,m\}$ be the set of $m$ classes, and for each (class) $j\in\Y$, let $\X_j\subset\R^n$ 
be the set of
points in the class $j$, assumed to be open with compact closure $\cX_j$,
Let $\X:=\bigcup_{j=1}^{m}\X_j\subset\R^n$ be the open set (with compact closure $\cX=\bigcup_{j=1}^{m}\cX_j$) of all points to be classified and 
$\mu$ be the joint probability distribution of $(\x,y)$ on $\cX\times \Y$. Write
\begin{equation}
\label{disintegration}
d\mu(\x,y)\,=\,\varphi(dy\vert\x)\,\phi(d\x)\,,\end{equation}
where $\mu$ has been disintegrated into its marginal $\phi$ on $\cX$ and 
its conditional probability distribution $\varphi(dy\vert\,\x)$ on $\Y$ given $\x\in\cX$. 
Next, each point $\x\in\X$ belongs to only one class and therefore $\X_i\cap\X_j=\emptyset$ for all pairs $(i,j)$ with $i\neq j$, and so we may and will
assume that $\phi(\cX_i\cap\cX_j)=0$ for all pairs $(i,j)$ with $i\neq j$.
Therefore one may write $\mu=\sum_{j=1}^{m}\mu_j$, with
\[d\mu_j(\x,y)\,=\,\delta_{\{j\}}(dy)\,\phi_j(d\x)\,,\quad j\in \Y\,,\]
for some marginals $\phi_j$ on $\cX_j$, $j\in \Y$. 
In particular:
\[\phi(A)\,=\,\mu(A\times\Y)\,=\,\sum_{j=1}^{m}\mu_j(A\times \Y)\,=\,\sum_{j=1}^{m}\phi_j(A)\,,\quad\forall A\in\mathcal{B}(\X)\,,\]
and therefore  $\phi=\sum_{j=1}^{m}\phi_j$.
Next, let $f:\cX\to\Y$ be the exact classifier
\begin{equation}
\label{exact-class}
\x\mapsto f(\x)\,=\,\left\{\begin{array}{l} \sum_{j=1}^{m}j\cdot 1_{\X_j}(\x)\mbox{ if $\x\in\X$,}\\ \mbox{$0$ otherwise.}\end{array}\right.
\end{equation}
So $f(\x)$ identifies the class of $\x\in\X$ and returns $f(\x)=0$ if $\x$ belong to some intersection $\cX_i\cap\cX_j$.
Notice that one may write
\begin{equation}
\label{mu-graph}
d\mu(\x,y)\,=\,\delta_{f(\x)}(dy)\,\phi(d\x)\,,\end{equation}
and so the joint distribution $\mu$ is supported on the graph $\mathbf{G}:=\{\,(\x,f(\x)): \x\in\cX\,\}$ of the function $f$.

The Christoffel function is a powerful tool from the theory of approximation and orthogonal polynomials and
one of its distinguishing features is its ability to identify the support 
of the underlying measure. So the CF $\Lambda^\mu_t$ associated with $\mu$ is an appropriate tool to approximate $f$ since the graph of $f$ is precisely the support of the measure $\mu$
in \eqref{mu-graph}. In Marx et al. \cite{constructive} the authors 
propose to approximate $f$ (when $\Vert f\Vert_\infty<M$) by:
\begin{equation}
\label{approx-constructive}
\hat{f}_t(\x)\,:=\,\arg\min_{y}\,\Lambda^{\mu+\varepsilon\mu_0}_t(\x,y)\,,\quad \forall\,\x\in\X\,,\end{equation}
with a small $\varepsilon>0$ and where $\mu_0$ is a measure with a density w.r.t. Lebesgue measure, positive on $\X\times [-M,M]$.
They prove nice theoretical convergence guarantees as $t$ increases; see \cite{constructive} for more details.
Notice that in the present supervised classification framework,  the function to approximate is a \emph{step function} so that the support of its graph is contained in a real algebraic variety, an even more specific case.

\section{Main result}
We first consider the ideal case of approximating the classifier $f$ in \eqref{exact-class} via the CF of the joint distribution $\mu$ in \eqref{mu-graph}. Then we next consider the more practical setting (as in machine learning) where we only have access to a finite sample (the training set).
In this case we use the \emph{empirical} measures $\phi_{j,N}$ associated with the points 
of the sample in class $``j"$. In section \ref{practical} we invoke
results from \cite{CD-2022} that relate the degree $t$ of the 
CF $\Lambda^{\phi_{j,N}}_t$ with  the size $N$ of the sample to ensure that important asymptotic properties of $\Lambda^{\phi_{j,N}}_t$ and $\Lambda^{\phi_j}_t$ as $t$ and $N$ increase, coincide.

\subsection{The CF on a real variety}
\label{real-variety}
Let $q\in\R[\x,y]$ be the polynomial $(\x,y)\mapsto v(\x,y)=\prod_{i=1}^m (y-i)$ and
let $\mu$ be the probability measure on $\bom=\cX\times\Y$ defined in \eqref{mu-graph}. 
Its support $\bom$ is contained in the real algebraic variety $V:=\{(\x,y): v(\x,y)=0\}=\R^n\times\Y$ and the ideal $\mathcal{I}=\langle v\,\rangle\subset \R[\x,y]$ generated by the polynomial $v$ is the ideal of polynomials that vanish on $V$. 

Observe that the moment matrix $\M_t(\mu)$ is singular since the vector $\mathbf{v}$ of coefficients of the polynomial $v$
is in the kernel of $\M_t(\mu)$ as soon as $t\geq m$. Indeed 
$\mathbf{v}^T\M_t(\mu)\mathbf{v}=\int v^2\,d\mu=0$
because the support of $\mu$ is contained in $V$. So 
the definition \eqref{def-christo-11} of the CF is not valid any more.
Denote by $L^2_t(\mu) \subset\R[\x,y]$ the space of polynomials on $V$ of total degree at most $t$
(and degree at most $m-1$ in the variable $y$), equipped with the inner product and norm inherited from $L^2(\bom,\mu)$. It turns out that $L^2_t(\mu)$ is a RKHS (Reproducing Kernel Hilbert Space). Then in this context, the variational definition \eqref{def-christo-2} of the CF associated with $\mu$ reads
\begin{equation}
\label{CF-V}
(\x,y)\mapsto \Lambda^\mu_t(\x,y)\,:=\,\inf_{p\in L^2_t(\mu)}\{\,\int_{\bom} p^2\,d\mu:\: p(\x,y)\,=\,1\,\}\,,\quad\forall (\x,y)\in\,V\,.\end{equation}
The set $\Gamma_t:=\{\,\x^{\balpha}\,y^k:\:k\leq m-1\,;\: \vert\balpha\vert+k\leq t\}\subset \R[\x,y]$
is a monomial basis of $L^2_t(\mu)$. 
Let $\M'_t(\mu)$ be the moment matrix associated with $\mu$ in \eqref{mu-graph} with rows and columns indexed by all monomials $(\x^{\balpha}\,y^k)$ of $\Gamma_t$ (and not all monomials $\x^{\balpha}\,y^k$ of total degree at most $t$), e.g.,
 listed according to the lexicographic ordering. Then
$\M'_t(\mu)$ is non singular and 
\[\Lambda^\mu_t(\x,y)\,=\,\v'_t(\x,y)^T\M'_t(\mu)^{-1}\,\v'_t(\x,y)\,,\quad\forall (\x,y)\,\in\,V\,,\]
where $\v'_t(\x,y)$ is the vector of all monomials of $\Gamma_t$. Alternatively
\[\Lambda^\mu_t(\x,y)\,=\,\sum_{(\balpha,k)\in\Gamma_t} \frac{1}{\lambda_{\balpha,k}}\,Q_{\balpha,k}(\x,y)^2\,,\]
where the $Q_{\balpha,k}$'s and the $\lambda_{\balpha,k}$'s are the eigenvectors and their respective eigenvalues 
associated with $\M'_t(\mu)$. Following \cite{constructive}, one may consider the perturbed measure $\mu+\varepsilon\mu_0$ 
where $\mu_0$ is a probability uniformly distributed on $\cX\times [0,m]$ and $\varepsilon>0$ is a small parameter.
Then 
\[\Lambda_t^{\mu+\varepsilon\mu_0}(\x,y)\,=\,\v_t(\x,y)^T\M_t(\mu+\varepsilon\mu_0)^{-1}\,\v_t(\x,y)\,,\quad\forall (\x,y)\,\in\,\R^n\times\R\,.\]
Recall that $\v_t(\x,y)$ is the vector of \emph{all} monomials $\x^{\balpha}\,y^k$ of degree at most $t$, and $\M_t(\mu+\varepsilon\mu_0)$ is the moment matrix of $\mu+\varepsilon\mu_0$ of order $t$, which is non singular for all $\varepsilon>0$. In addition $\Lambda^{\mu+\varepsilon\mu_0}_t$ is defined for \emph{all} $(\x,y)\in\R^{n+1}$
whereas $\Lambda^\mu_t$ in \eqref{CF-V} is defined for all $(\x,y)\in V$. Then as proved in \cite{constructive}
the classifier $\hat{f}_t$ in \eqref{approx-constructive} approximates $f$ as $t$ increases. Next we show that by 
a slightly change of the vector space $L^2_t(\mu)$ in \eqref{CF-V}, the resulting CF has nice additional properties
that can be exploited to provide the resulting classifier \eqref{approx-constructive} with a clear and more intuitive interpretation.

\paragraph{A slight variant of the CF}

We now introduce a slight variant $\hat{\Lambda}^\mu_t$ of the CF $\Lambda^\mu_t$, defined by:
\begin{equation}
\label{CF-V-variant}
(\x,y)\mapsto \hat{\Lambda}^\mu_t(\x,y)\,:=\,\inf_{p\in \L^2_t(\mu)}\{\,\int_{\bom} p^2\,d\mu:\: p(\x,y)\,=\,1\,\}\,,\quad\forall (\x,y)\in\,V\,,\end{equation}
where $\L^2_t(\mu)=:\R[\x,y]_{t,m-1}$ is the vector space of polynomials of degree at most $t$ with respect to the variable $\x$ and at most $m-1$ with respect to the variable $y$, so that $L^2_t(\mu)\subset\,\L^2_t(\mu)\subset L^2_{t+m-1}(\mu)$. 
\begin{proposition}
\label{prop-1}
Let $\Lambda^\mu_t$ and $\hat{\Lambda}^\mu_t$ be as in \eqref{CF-V} and \eqref{CF-V-variant} respectively. Then:
\begin{equation}
\label{prop-1-1}
\Lambda^\mu_{t+m-1}(\x,y)\,\leq\,\hat{\Lambda}^\mu_{t}(\x,y)\,\leq\,\Lambda^\mu_{t}(\x,y)\,,\quad\forall (\x,y)\,\in\,V\,.\end{equation}
\end{proposition}
\begin{proof} Follows from $L^2_t(\mu)\subset\,\L^2_t(\mu)\subset L^2_{t+m-1}(\mu)$ and the definitions
\eqref{CF-V} and \eqref{CF-V-variant}.
\end{proof}
Proposition \ref{prop-1} states that $\hat{\Lambda}^\mu_t$ and $\Lambda^\mu_t$ are close but as we next see, $\hat{\Lambda}^\mu_t$ has an interesting additional feature. Namely, it has a nice characterization in closed form which when exploited for classification
leads to a classifier with a clear interpretation.
Let $(\theta_j)_{j\in [m]}\subset\R[y]_{m-1}$ be the interpolation polynomials at the points $\{1,2,\ldots,m\}$ of $\Y$, i.e.,
\[y\mapsto \theta_j(y)\,:=\,\frac{\prod_{i\neq j}(y-i)}{\prod_{i\neq j}(j-i)}\,,\quad i=1,\ldots,m\,,\]
which form an orthonormal family with respect to the uniform probability measure on $\Y$.

\begin{theorem}
\label{theo-1}
For each $j\in\Y$, let $(P^j_{\balpha})_{\balpha\in\N}\subset\R[\x]$ be a family of polynomials
that are orthonormal with respect to the marginal probability measure $\phi_j$ of $\mu_j$, and
let $\Lambda^{\phi_j}_t$ be the standard Christoffel function associated with $\phi_j$ on $\cX_j$.
Then :

(i) The family $(\theta_j(y)\,P^j_{\balpha}(\x))_{\balpha\in\N^n_t}\subset\R[\x,y]$ 
is an orthonormal basis of $\L^2_t(\mu)$.

(ii) The Christoffel function $\hat{\Lambda}^\mu_t$ defined in \eqref{CF-V-variant} satisfies
\begin{eqnarray}
\label{christo-formula-compact-multi-0}
\hat{\Lambda}^\mu_t(\x,y)^{-1}&=&
\sum_{j\in \Y}\theta_j(y)^2\,\sum_{\balpha\in\N^n_t}P^j_{\balpha}(\x)^2\,,\quad\forall (\x,y)\in\R^n\times\Y\\
\label{christo-formula-compact-multi-1}
&=&\sum_{j\in\Y}\theta_j(y)^2\,\Lambda^{\phi_j}_t(\x)^{-1}\,,\quad\forall (\x,y)\in\R^n\times\Y\,\\
\label{christo-formula-compact-multi-2}
&=&\sum_{j\in \Y}\delta_{y=j}\,\Lambda^{\phi_j}_t(\x)^{-1}\,,\quad\forall (\x,y)\in\R^n\times\Y\,.
\end{eqnarray}
\end{theorem}

 \begin{proof}
 (i)  Every element of $\L^2_t(\mu)=\R[\x,y]_{t,m-1}$ is of the form
 $\sum_{j=0}^{m-1}y^j\,q_j(\x)$ with $q_j\in\R[\x]_t$. Hence consider a polynomial $u\in\L^2_t(\mu)$ in the form
 $u(\x,y):=p(y)\,q(\x)$ 
 for some $q\in\R[\x]_t$ and some $p\in\R[y]_{m-1}$, arbitrary. 
 Then as $(P^j_{\balpha})_{\balpha\in\N^n_t}$ generates $\R[\x]_t$, observe that 
 for every $j\in \Y$: 
 \[ q(\x)\,=\,\sum_{\balpha\in\N^n_t}q^j_{\balpha}\,P^j_{\balpha}(\x)\,,\quad\forall \x\in\R^n\,,\]
  for some coefficients $(q^j_{\balpha})_{\balpha\in\N^n_t}$.
 Next, as the polynomials $(\theta_j)_{j\in \Y}$ generate $\R[y]_{m-1}$, write
$p(y)=\sum_{j\in\Y} p_j\,\theta_j(y)$  for some coefficients $(p_j)_{j\in \Y}$, and therefore
 \[ u(\x,y)\,=\,\sum_{j\in \Y} p_j\,(\theta_j(y)\,q(\x))\,=\,\sum_{\balpha\in\N^n_t,\,j\in \Y} p_j\,q^j_{\balpha}\,\left[\theta_j(y)\,P^j_{\balpha}(\x)\,\right]\,.\]
 \emph{Orthogonality.}
 If $i\neq j$ then $\theta_i(y)\theta_j(y)=0$ everywhere on the support of $\mu$ and therefore
 \[ \int_{\bom}\theta_i(y)\,P^i _{\balpha}(\x)\,\theta_j(y)\,P^j_{\bbeta}(\x)\,d\mu(\x,y)\,=\,0\,,\]
whereas if $i=j$ then
\begin{eqnarray*}
 \int_{\bom}\theta_i(y)^2\,P^i _{\balpha}(\x)\,P^i_{\bbeta}(\x)\,d\mu(\x,y)&=&
\sum_{j=1}^{m}\int_{\bom}\theta_i(y)^2\,P^i _{\balpha}(\x)\,P^i_{\bbeta}(\x)\,d\mu_j(\x,y)\\
&=&\int_{\bom}\theta_i(y)^2\,P^i _{\balpha}(\x)\,P^i_{\bbeta}(\x)\,d\mu_i(\x,y)\,=\,
\int_{\bom}P^i _{\balpha}(\x)\,P^i_{\bbeta}(\x)\,d\phi_i(\x)\,=\,\delta_{\balpha=\bbeta}\,.
\end{eqnarray*}
 Next, as $\L^2_t(\mu)\subset\R[\x,y]_{t,m-1}$,   its cardinality is
 $r(t):=m\cdot {n+t\choose t}$ which is also  the number of terms in the family $(\theta_j(y)\,P^j_{\balpha}(\x))_{\balpha\in\N^n_t,j\in \Y}$
 which also generates $\L^2_t(\mu)$. 
 Hence $(\theta_j(y)\,P^j_{\balpha}(\x))_{\balpha\in\N^n_t,j\in\Y}$ is an orthonormal basis of $\L^2_t(\mu)$. 
 
  (ii) Let $\tM_t(\mu)$ be the moment matrix of degree $t$ with rows and columns indexed by
  monomials $(\x^{\balpha}\,y^j)_{\balpha\in\N^n_t,0\leq j\leq m-1}$ (e.g. with lexicographic ordering), and let
  $\tv_t(\x,y)$ be the vector of monomials $\x^{\balpha}\,y^k$ listed with the same ordering.
 Observe that \eqref{CF-V-variant} reads
 \begin{equation}
 \label{pb-convex}
 \hat{\Lambda}^\mu_t(\x,y)\,=\,\min_\p\,\{\,\p^T\tM_t(\mu)\,\p:\: \langle \p,\tv_t(\x,y)\rangle\,=\,1\,\}\,,\end{equation}
 where $\p\in\R^{r(t)}$ is the vector of coefficients of $p\in\L^2_t(\mu)$ in that basis. Then 
 \eqref{pb-convex} is a convex optimization problem whose optimal solution $\p^*$ satisfies
 $2\,\hat{\M}_t(\mu)\,\p^*=\lambda^*\,\tv_t(\x,y)$
 for some scalar $\lambda^*$. Hence $\lambda^*=2\hat{\Lambda}^\mu_t(\x,y)$ and 
 \[\p^*\,=\,\hat{\Lambda}^\mu_t(\x,y)\,\tM_t(\mu)^{-1}\,\tv_t(\x,y)\,\]
 so that the corresponding polynomial $p^*\in\L^2_t(\mu)$ reads
 \begin{eqnarray*}
 p^*(\u,z)&=&\hat{\Lambda}^\mu_t(\x,y)\,
 \tv_t(\u,z)^T\tM_t(\mu)^{-1}\,\tv_t(\x,y)\\
 &=&\hat{\Lambda}^\mu_t(\x,y)\,
 \sum_{\balpha\in\N^n_t,j\in \Y}\theta_j(y)\,\theta_j(z)\,P^j _{\balpha}(\u)\,P^j _{\balpha}(\x)\,,\quad (\u,z)\in\R^n\times \Y\,,\end{eqnarray*}
  and therefore
  \[1\,=\,p^*(\x,y)\,=\,\hat{\Lambda}^\mu_t(\x,y)\,\sum_{\balpha\in\N^n_t,j\in \Y}\theta_j(y)^2\,P^j _{\balpha}(\x)^2\,,\quad \forall (\x,y)\in\R^n\times \Y\,,\]
  which is \eqref{christo-formula-compact-multi-0}. In particular we also retrieve that
  \begin{equation}
  \label{alternative}
  \hat{\Lambda}^\mu_t(\x,y)^{-1}\,=\,\hat{\v}_t(\x,y)^T\,\hat{\M}_t(\mu)^{-1}\,\hat{\v}_t(\x,y)\,,\quad \forall (\x,y)\,\in\,V\,.
  \end{equation}
  Next \eqref{christo-formula-compact-multi-1} follows from
  the definition of the Christoffel function associated with $\phi_j$ for each $j\in\Y$, and \eqref{christo-formula-compact-multi-2} follows from   the properties of interpolation polynomials $(\theta_j)_{j\in\Y}$.
 \end{proof}
So whenever $y\in\Y$, the Christoffel function 
$\hat{\Lambda}^\mu_t(\x,y)$ has a very simple expression \eqref{christo-formula-compact-multi-1}, stated directly in terms of the 
Christoffel functions $(\Lambda^{\phi_j}_t(\x))_{j=1,\ldots,m}$ associated with the classes $j=1,\ldots,m$. This is quite natural but is proper to 
the CF $\hat{\Lambda}^\mu_t$ and not to the standard CF $\Lambda^\mu_t$.\\

\paragraph{An ideal classifier}
Given the Christoffel function $\hat{\Lambda}^\mu_t$ defined in \eqref{christo-formula-compact-multi-1} and 
inspired by \eqref{approx-constructive}, a natural candidate classifier is the function
\begin{equation}
\label{ideal-classifier}
\x\mapsto \hat{f}_t(\x)\,:=\,\arg\min_{y\in\Y}\,\hat{\Lambda}^\mu_t(\x,y)^{-1}\,=\,\arg\max_{y\in\Y}\,\hat{\Lambda}^\mu_t(\x,y),\quad\forall \x\in \X\,,
\end{equation}
which in view of \eqref{christo-formula-compact-multi-2} reads:
\begin{equation}
\label{ideal-classifier-1}
\x\mapsto \hat{f}_t(\x)\,:=\,\arg\max_{k\in [m]}\,\Lambda^{\phi_k}_t(\x)\,,\quad\forall \x\in \X\,.
\end{equation}
Observe that the ``$\max$" in \eqref{ideal-classifier-1} is over $y\in [m]$ and not over the interval  $[0,m]$. This is because
in supervised classification, we know that $f(\x)\in [m]$ for all $\x\in\X$.
The rationale is the following: Let $\x\in\X_j$ be fixed arbitrary, so that $\x\not\in\cX_k$ for every $k\neq j$. 
As $t$ increases, $\Lambda_t^{\phi_k}(\x)$ decreases to zero
exponentially while $\Lambda^{\phi_j}_t(\x)$ decreases not faster than $t^n$. Therefore for $t$ sufficiently large,
necessarily $\Lambda^{\phi_k}_t(\x)<\Lambda^{\phi_j}_t(\x)$ for all $k\neq j$, and so by \eqref{christo-formula-compact-multi-2}, 
\[\x\in\X_j\Rightarrow\exists t_0\mbox{ s.t. }\hat{f}_t(\x)\,=\,\arg\min_{y\in\Y}\,\hat{\Lambda}^\mu_t(\x,y)^{-1}
\,=\,\arg\max_{k}\,\Lambda^{\phi_k}_t(\x)\,=\,j\,,\quad\forall t\geq t_0\,.\]
An even stronger  \emph{almost-uniform} result holds. Let $\partial \X_j$ denote the boundary 
of the set $\X_j\subset\R^n$.
\begin{theorem}
\label{th-main}
Let $\hat{f}$ be as in \eqref{ideal-classifier-1}
and  let $\X_j^\varepsilon:=\{\,\x\in\X_j: d(\x,\partial \X_j)>\varepsilon\,\}$ where $\varepsilon>0$ is fixed. 
Assume that for every  $j\in [m]$, $\phi_j$ has a density w.r.t. Lebesgue measure $\lambda$ restricted to $\X_j$, bounded from below by $c>0$.
Then there exists $t_\varepsilon$ such that $\hat{f}_t(\x)=j$ for all $\x\in\X_j^\varepsilon$ and all $t\geq t_\varepsilon$.
\end{theorem}
\begin{proof}
Let $s(t):={n+t\choose t}$ and denote by $\mathrm{diam}(S)$ the diameter of a bounded set $S\subset\R^n$. 
Let $\gamma_k:=\phi_k(\X_k)$, $k\in [m]$.
If $\x\in \X^\varepsilon_j$ then $d(\x,\cX_k)>\varepsilon$ for all $k\neq j$. Hence by 
\cite[Lemma 6.6]{adv-comp} (and using that $\Lambda_t^{\phi_k/\gamma_k}=\Lambda_t^{\phi_k}/\gamma_k$),
\[\frac{\gamma_k}{s(t)}\,\Lambda^{\phi_k}_t(\x)^{-1}\,\geq\,2^{\frac{t\,\varepsilon}{\varepsilon+{\rm diam}(\X_k)}-3}\,t^{-n}\,(\frac{n}{e})^n\,\exp(-n^2/t)\,,\quad\forall \x\in\X_j^\varepsilon\,.\]
On the other hand, as $d(\x,\partial \X_j)>\varepsilon$, we can invoke
\cite[Lemma 6.2]{adv-comp} extended to measures with density w.r.t. Lebesgue
bounded from below by $c>0$ (see 
\cite[Assumption 3.11]{adv-comp}) and use
$(t+1)(t+2)(t+3)/((n+t+1)(n+t+2)(n+2t+6))\geq 1/2(n+1)^3$, to obtain
\[\frac{\gamma_j}{s(t)}\,\Lambda^{\phi_j}_t(\x)^{-1}\,\leq\,\frac{2\,\lambda(\X_j)}{c\,\varepsilon^n \omega_n}\,(1+n)^3\,,\quad\forall \x\in\,\X^\varepsilon_j\,,\]
where $\omega_n$ is the $n$-dimensional area of $\mathbb{S}^{n+1}$. Hence clearly there exists $t_\varepsilon$ such that
$\Lambda^{\phi_k}_t(\x)^{-1}>\Lambda^{\phi_j}_t(\x)^{-1}$ for all $k\neq j$ and all $\x\in\X^\varepsilon_j$,
whenever $t\geq t_\varepsilon$. In particular, as the functions $\Lambda^{\phi_k}_t$ are strictly positive, continuous and $\overline{\X}_j$ is compact, there exists $a_j>0$ such that
\begin{equation}
\label{aux-1}
 \forall \x\in\X^{\varepsilon}_j\,:\quad
 \Lambda^{\phi_k}_t(\x)\,<\,\Lambda^{\phi_j}_t(\x)- a_j\,,\quad\forall k\neq j\,.
\end{equation}
Therefore the result follows from \eqref{christo-formula-compact-multi-2} and \eqref{ideal-classifier}.
\end{proof}

\subsection{Application to supervised classification with noiseless deterministic labels}
\label{practical}
In supervised classification we do not have access to the CF $\Lambda^\mu_t$ or $\hat{\Lambda}^\mu_t$.
We only have access to a sample of $N$ points ${\rm Tr}_N=\{\,(\x(i),y(i)): i=1,\ldots,N\}\subset\X$ (the training data set) and a sample of test points (the test data set).  For instance,
in a typical Machine Learning (ML) approach one tries to \emph{learn} a classifier function 
$f$ in \eqref{exact-class} from the supervised data ${\rm Tr}_N$ by computing parameters of a deep neural network
that minimize some loss function. Usually, the number of parameters is very large (compared to $N$) 
making the resulting solution sensitive to a classical \emph{overfitting} phenomenon. One way to attenuate  this 
overfitting phenomenon is to add an appropriate regularization term to the loss function in the criterion to minimize.

One reason behind this overfitting phenomenon is that in minimizing the loss function,
each data point $(\x(i),y(i))$ is treated \emph{separately}. Ideally one should somehow consider the entire \emph{training} set ${\rm Tr}_N$ itself and not its members separately. This is precisely what the CF function approach does. Indeed the training set 
${\rm Tr}_N$ is used to construct the empirical (discrete) analogues $\phi_{j,N}$ of the measures $\phi_j$,
 and their associated empirical Christoffel function $\Lambda^{\phi_{j,N}}_t$, now obtained from empirical moments. Remarkably, even though the geometry of the support of $\phi_{j,N}$ is quite trivial, the CF $\Lambda^{\phi_{j,N}}_t$
is still close to $\Lambda^{\phi_j}_t$ in a certain sense and the training set ${\rm Tr}_N$
can still be used to infer properties of the underlying measures $\phi_j$.
Hence, and importantly, even though the mathematical object $\Lambda^{\phi_{j,N}}_t$ 
is built from \emph{individual} items, it is in fact concerned with the \emph{cloud} of points of $\mathrm{Tr}_N$ in class $\{j\}$, rather than the points $\x(i)$ in that class
taken separately.
However of course,  for $\Lambda^{\phi_{j,N}}_t$ to recover asymptotic properties of $\Lambda^{\phi_{j}}_t$, the sample size $N$ and the degree $t$ cannot be chosen independently; see \cite{CD-2022,adv-comp} for more details.

\paragraph{Setting}

For every $j\in [m]$, let $\mathrm{Tr}^j_N=(\x(i))_{i\leq N}\subset\X_j$ be a training set for class $\{j\}$,
where $\x(i)$ are i.i.d. random vectors with common distribution $\phi_j$ whose support is $\X_j$.
So the whole training set $\mathrm{Tr}_N$ has a total of $mN$ points
where the $N$ points in each class $j$ are sampled from $\phi_j$.
For every fixed $t$, 
a natural approach suggested by \eqref{ideal-classifier-1}, 
consists of:
\begin{itemize}
\item Computing the Christoffel function $\Lambda^{\phi_{j,N}}_t$ associated with the empirical prob. measure
\begin{equation}
\label{empirical-j}
\phi_{j,N}\,:=\,\frac{1}{N}\sum_{\x(i)\in \X_j}\delta_{\{\x(i)\}}\,,\quad\forall j\,\in\,[m]\,.\end{equation}
Following \eqref{def-christo-11}, $\Lambda^{\phi_{j,N}}_t(\x)=\v_t(\x)^T\M_t(\phi_{j,N})^{-1}\,\v_t(\x)$,
for all $\x\in\R^n$ and all $j\in [m]$. The moments of $\phi_{j,N}$ are easily obtained by
\[\phi_{j,N}(\balpha)\,:=\,\{\,\frac{1}{N}\,\sum_{i}\x(i)^{\balpha}\,:\:\x(i)\in\mathrm{Tr}^j_N\,\}\,,\quad\forall \balpha\in\N^n\,,\]
and the moment matrix $\M_t(\phi_{j,N})$ is nonsingular 
for sufficiently large $t$.
\item Following \eqref{ideal-classifier-1}, introduce the empirical classifier
\begin{equation}
\label{empirical-classifier}
\x\mapsto \hat{f}^N_t(\x)\,:=\,\arg\max_{k}\,\Lambda^{\phi_{k,N}}_t(\x)\,,\quad\forall \x\in \X\,.
\end{equation}
\end{itemize}
Then the empirical version of Theorem \ref{th-main} reads as follows:

\begin{theorem}
\label{th-final}
For every $j\in [m]$, let $(\x(i))_{i\leq N}\subset\X_j$ be i.i.d. random vectors according to a distribution $\phi_j$ whose support is $\cX_j$ and 
which satisfies the assumption in Theorem \ref{th-main}.
Let $\phi_{j,N}$ be as in \eqref{empirical-j}, and $\hat{f}^N_t$ be as in \eqref{empirical-classifier}.
Given $\varepsilon>0$ fixed, let $\X_j^\varepsilon:=\{\,\x\in\X_j: d(\x,\partial \X_j)>\varepsilon\,\}$, $j\in [m]$.
Then there exists $t_\varepsilon$ such that for all $t>t_\varepsilon$ fixed,
with probability $1$ (with respect to the random samples $\mathrm{Tr}^k_N\subset\X_k$, $k\in [m]$),
 $\hat{f}^N_{t}(\x)=j$ for all $\x\in\X_j^\varepsilon$, for sufficiently large $N$.
\end{theorem}
\begin{proof}
  By Theorem \ref{th-main} there exist $t_\varepsilon$ and $a_j>0$, such that $\Lambda_t^{\phi_j}(\x)-\Lambda^{\phi_k}_t(\x)>a_j$  for all $k\neq j$, all $t\geq t_\varepsilon$ and all $\x\in\X^\varepsilon_j$; see \eqref{aux-1}. On the other hand,
  with $t$ fixed, by \cite[Theorem 3.13]{adv-comp}, 
 \[\mbox{For every $j\in [m]$:}\quad
 \sup_{\x\in\R^n}\,\vert\,\Lambda^{\phi_{j,N}}_t(\x)-\Lambda^{\phi_j}_t(\x)\,\vert\,\stackrel{a.s.}{\to}\,0\,,\quad\mbox{as $N\to\infty$\,,}\]
 where the \emph{``a.s."} is with respect to the random sample $\mathrm{Tr}^j_N\subset\X_j$.
Hence for every $t>t_\varepsilon$ fixed, with probability $1$ (with respect to the $m$ random 
samples $\mathrm{Tr}^k_N\subset\X_k$, $k\in [m]$)
\[\forall\x\in\X^{\varepsilon}_j\,,\:\forall k\neq j\,:\quad
\Lambda^{\phi_{k,N}}_t(\x)\,<\,\Lambda^{\phi_{j,N}}_t(\x)-a_j/2\,,\]
for sufficiently large $N$, and therefore $\hat{f}^N_{t}(\x)=j$ for all $\x\in\X^{\varepsilon}_j$, $j\in [m]$.
 \end{proof}
 
 \subsection{The general case where the supports  $\X_j$ are not disjoint}
 We next briefly consider the case where  the assumption $\X_i\cap\X_j=\emptyset$ for all $i\neq j$, is \emph{not} satisfied. That is,  misclassifications occur or some points $\x\in\X$ may indeed belong to several classes, as can be the case in some practical situations.

So let $\mu$ on $\cX\times \Y$, $\varphi(dy\vert\cdot)$ on 
$\Y$ and $\phi$ on $\cX$ be as in  \eqref{disintegration}, but now with possibly 
  $\X_i\cap\X_j\neq\emptyset$ for some $i\neq j$. For each $j\in [m]$, introduce the measures
  $\phi_j$ on $\X$, $j\in [m]$, defined by:
 \begin{equation}
 \label{phiji}
\phi_j(d\x)\,:=\, \varphi(j\,\vert\,\x)\,\phi(d\x)\,,\quad  j\in \Y\,.
 \end{equation}
  So in contrast to section \ref{real-variety}, $\phi_j(\X_i)\neq0$ is allowed for $i\neq j$.
  Then an analogue of Theorem \ref{theo-1} reads:
  \begin{theorem}
\label{theo-2}
For each $j\in\Y$, let $(P^j_{\balpha})_{\balpha\in\N}\subset\R[\x]$ be a family of polynomials
that are orthonormal with respect to the measure $\phi_j$ on $\cX$ defined in \eqref{phiji}, and
let $\Lambda^{\phi_j}_t$ be the standard Christoffel function associated with 
$\phi_j$. Then :

(i) The family $(\theta_j(y)\,P^j_{\balpha}(\x))_{\balpha\in\N^n_t}\subset\R[\x,y]$ 
is an orthonormal basis of $\L^2_t(\mu)$.

(ii) The Christoffel function $\hat{\Lambda}^\mu_t$ defined in \eqref{CF-V-variant} satisfies
\begin{eqnarray}
\label{theo-2-1}
\hat{\Lambda}^\mu_t(\x,y)^{-1}&=&
\sum_{j\in \Y}\theta_j(y)^2\,\sum_{\balpha\in\N^n_t}P^j_{\balpha}(\x)^2\,,\quad\forall (\x,y)\in\R^n\times\Y\\
\label{theo-2-2}
&=&\sum_{j\in\Y}\theta_j(y)^2\,\Lambda^{\phi_j}_t(\x)^{-1}\,=\,
\sum_{j\in [m]}\delta_{y=j}\,\Lambda^{\phi_j}_t(\x)^{-1}\,,\quad\forall (\x,y)\in\R^n\times\Y\,.
\end{eqnarray}
\end{theorem}
 \begin{proof}
 (i) If $i\neq j$ then $\theta_i(y)\theta_j(y)=0$ everywhere on the support of $\mu$ and therefore
 \[ \int_{\bom}\theta_i(y)\,P^i _{\balpha}(\x)\,\theta_j(y)\,P^j_{\bbeta}(\x)\,d\mu(\x,y)\,=\,0\,,\]
whereas if $i=j$ then
\begin{eqnarray*}
 \int_{\bom}\theta_i(y)^2\,P^i _{\balpha}(\x)\,P^i_{\bbeta}(\x)\,d\mu(\x,y)&=&
  \int_{\X}P^i _{\balpha}(\x)\,P^i_{\bbeta}(\x)\,\left(\int_{\Y}\theta_i(y)^2\varphi(dy\vert\x)\right)\,\phi(d\x)\\
&=&\int_{\X}P^i _{\balpha}(\x)\,P^i_{\bbeta}(\x)\,\left(\sum_{j\in [m]}\theta_i(j)^2\varphi(j\,\vert\,\x)\right)\phi(d\x)\\
&=&\int_{\X}P^i _{\balpha}(\x)\,P^i_{\bbeta}(\x)\,\phi_i(d\x)\,=\,\delta_{\balpha=\bbeta}\,.
\end{eqnarray*}
(ii) The rest of the proof is similar to that of Theorem \ref{theo-1}.
\end{proof}
By \eqref{theo-2-2}, $\Lambda^{\phi_j}_t$ is easily obtained 
as $\hat{\Lambda}^\mu_t(\cdot,j)$, $j\in [m]$, and 
$\hat{\Lambda}^\mu_t(\x,y)$ is in turn obtained for instance via \eqref{alternative} from 
the moment matrix $\hat{\M}_t(\mu)$ of the joint distribution $\mu$.
As the CF is an appropriate tool for support inference, notice that intersections of 
super level sets $\G_{i,\gamma}\cap\G_{j,\gamma}$, with $\G_{i,\gamma}:=\{\x:\Lambda^{\phi_i}_t(\x)\geq\gamma\}$, should provide indications on whether $\X_i\cap\X_j=\emptyset$ if $i\neq j$.

Once again the CF $\hat{\Lambda}^\mu_t$ has the simple and nice expression
\eqref{theo-2-2} only in terms of the CF's $\Lambda^{\phi_j}_t$, which suggests to define
a classifier $\hat{f}_t(\x)$ exactly as in \eqref{ideal-classifier}. 
The only difference is the meaning of $\phi_j$ and its implications. For instance if $\X_i\cap\X_j\neq\emptyset$ then 
a point $\x\in \X$ can belong to ${\rm supp}(\phi_i)\cap{\rm supp}(\phi_j)$ with $i\neq j$, and therefore the two scores $\Lambda^{\phi_i}_t(\x)$ and $\Lambda^{\phi_j}_t(\x)$ can be comparable even for large $t$, whereas before for sufficiently large $t$, the score $\Lambda^{\phi_j}_t(\x)$ (with $j=\mathrm{class}(\x)$) clearly dominates all other scores. In particular there is no analogue of Theorem \ref{th-main}. Evaluating how efficient the resulting classifier
$\hat{f}_t$ can be in a practical empirical context of 
sampled data as in section \ref{practical}, is beyond the scope of the present note.

\section{Conclusion}
The Christoffel function can provide a simple tool in supervised classification, with 
some theoretical guarantees. However to obtain $\Lambda^{\phi_{j,N}}_t$ \emph{explicitly} one must
handle matrices of size ${n+t\choose n}$ (inversion or eigenvectors) with a computational cost 
$O(t^n)$ that grows rapidly with $t$. Therefore so far, in this form this tool is limited to small dimension problems.
On the other hand, evaluation of $\Lambda^{\phi_{j,N}}_t(\boldsymbol{\xi})$ 
at a point $\boldsymbol{\xi}\in\R^n$ via \eqref{CF-V} only requires to solve  a simple convex quadratic 
optimization problem, which can be done efficiently even for large $n$. Finally a detailed 
analysis of possible learning rates that can be obtained with this method remains to be done.


\begin{thebibliography}{las}
\bibitem{data-driven}
Brunton S.L. and Kutz J.N.
\emph{Data-Driven Science and Engineering: Machine Learning, Dynamical Systems, and Control},
Cambridge University Press, Cambridge, UK, 2019.
\bibitem{neurips}
Lasserre J.~B., Pauwels E.
\emph{Sorting out typicality via the inverse moment matrix {S}{O}{S} polynomial},
in \emph{Advances in Neural Information Processing Systems},
D.D. Lee, M. Sugiyama, U.V. Luxburg, I. Guyon and R. Garnett Eds., 
Curran Associates, Inc., pp. 190--198, 2016.
\bibitem{CD-2022}
Lasserre J.~B., Pauwels E., Putinar M. \emph{The Christoffel-Darboux Kernel for Data Analysis},
Cambridge Monographs on Applied and Computational Mathematics,
Cambridge University Press, Cambridge, UK, 2022.
\bibitem{adv-comp}
Lasserre J.B, Pauwels E. \emph{The empirical Christoffel  function with applications in data analysis},
Adv. Comput. Math. {\bf  45}, pp. 1439--1468, 2019.
\bibitem{constructive}
Marx S., Pauwels E., Weisser T., Henrion D., Lasserre J.~B.
\emph{Semi-algebraic approximation using Christoffel-Darboux kernel}, Constr. Approx. {\bf 54}, pp. 391--429, 2021.
\bibitem{FoCM}
Pauwels E., Putinar M., Lasserre J.~B.
\emph{Data analysis from empirical moments and the Christoffel function},
Found. Comput. Math. {\bf 21}, pp. 243--273, 2021.


\end{thebibliography}
\end{document}